\documentclass[11pt]{amsart}

\usepackage{amsmath}
\usepackage{amssymb}
\usepackage{amscd}
\usepackage[usenames, dvipsnames]{color}
\usepackage{hyperref}
\usepackage{mathrsfs}
\usepackage{eucal}
\usepackage{upgreek}
\usepackage[makeroom]{cancel}
\usepackage[normalem]{ulem}
\usepackage{array}
\usepackage{verbatim}
\usepackage{mathtools}
\usepackage{stmaryrd}

\usepackage{enumitem}

\usepackage{xy}
\xyoption{all}
\usepackage{tikz-cd}

\newcommand{\nc}{\newcommand}

\nc{\md}{\operatorname{-}}

\renewcommand{\AA}{{\mathbb{A}}}
\nc{\CC}{{\mathbb{C}}}
\nc{\DD}{{\mathbb{D}}}
\nc{\LL}{{\mathbb{L}}}
\nc{\RR}{{\mathbb{R}}}
\renewcommand{\P}{{\mathbb{P}}}
\nc{\OO}{{\mathbb{O}}}
\renewcommand{\SS}{{\mathbb{S}}}
\nc{\QQ}{{\mathbb{Q}}}
\nc{\ZZ}{{\mathbb{Z}}}
\nc{\Z}{{\mathbb{Z}}}

\nc{\cA}{{\mathcal{A}}}
\nc{\cB}{{\mathcal{B}}}
\nc{\cC}{{\mathcal{C}}}
\nc{\cD}{{\mathcal{D}}}
\nc{\cE}{{\mathcal{E}}}
\nc{\cF}{{\mathcal{F}}}
\nc{\cG}{{\mathcal{G}}}
\nc{\cH}{{\mathcal{H}}}
\nc{\cI}{{\mathcal{I}}}
\nc{\cJ}{{\mathcal{J}}}
\nc{\cK}{{\mathcal{K}}}
\nc{\cL}{{\mathcal{L}}}
\nc{\cM}{{\mathcal{M}}}
\nc{\cN}{{\mathcal{N}}}
\nc{\cO}{{\mathcal{O}}}
\nc{\cP}{{\mathcal{P}}}
\nc{\cQ}{{\mathcal{Q}}}
\nc{\cR}{{\mathcal{R}}}
\nc{\cS}{{\mathcal{S}}}
\nc{\cT}{{\mathcal{T}}}
\nc{\cU}{{\mathcal{U}}}
\nc{\cV}{{\mathcal{V}}}
\nc{\cW}{{\mathcal{W}}}
\nc{\cX}{{\mathcal{X}}}
\nc{\cY}{{\mathcal{Y}}}
\nc{\cZ}{{\mathcal{Z}}}

\nc{\rc}{{\mathrm{c}}}
\nc{\rd}{{\mathrm{d}}}
\nc{\rf}{{\mathrm{f}}}
\nc{\rh}{{\mathrm{h}}}
\nc{\rs}{{\mathrm{s}}}
\nc{\rch}{{\mathrm{ch}}}
\nc{\rtd}{{\mathrm{td}}}

\nc{\rA}{{\mathrm{A}}}
\nc{\rB}{{\mathrm{B}}}
\nc{\rC}{{\mathrm{C}}}
\nc{\rD}{{\mathrm{D}}}
\nc{\rE}{{\mathrm{E}}}
\nc{\rF}{{\mathrm{F}}}
\nc{\rG}{{\mathrm{G}}}
\nc{\rH}{{\mathrm{H}}}
\nc{\rI}{{\mathrm{I}}}
\nc{\rJ}{{\mathrm{J}}}
\nc{\rK}{{\mathrm{K}}}
\nc{\rL}{{\mathrm{L}}}
\nc{\rM}{{\mathrm{M}}}
\nc{\rN}{{\mathrm{N}}}
\nc{\rO}{{\mathrm{O}}}
\nc{\rP}{{\mathrm{P}}}
\nc{\rQ}{{\mathrm{Q}}}
\nc{\rR}{{\mathrm{R}}}
\nc{\rS}{{\mathrm{S}}}
\nc{\rT}{{\mathrm{T}}}
\nc{\rU}{{\mathrm{U}}}
\nc{\rV}{{\mathrm{V}}}
\nc{\rW}{{\mathrm{W}}}
\nc{\rX}{{\mathrm{X}}}
\nc{\rY}{{\mathrm{Y}}}
\nc{\rZ}{{\mathrm{Z}}}

\nc{\bA}{{\mathbf{A}}}
\nc{\bB}{{\mathbf{B}}}
\nc{\bC}{{\mathbf{C}}}
\nc{\bD}{{\mathbf{D}}}
\nc{\bE}{{\mathbf{E}}}
\nc{\bF}{{\mathbf{F}}}
\nc{\bG}{{\mathbf{G}}}
\nc{\bH}{{\mathbf{H}}}
\nc{\bI}{{\mathbf{I}}}
\nc{\bJ}{{\mathbf{J}}}
\nc{\bK}{{\mathbf{K}}}
\nc{\bL}{{\mathbf{L}}}
\nc{\bM}{{\mathbf{M}}}
\nc{\bN}{{\mathbf{N}}}
\nc{\bO}{{\mathbf{O}}}
\nc{\bP}{{\mathbf{P}}}
\nc{\bQ}{{\mathbf{Q}}}
\nc{\bR}{{\mathbf{R}}}
\nc{\bS}{{\mathbf{S}}}
\nc{\bT}{{\mathbf{T}}}
\nc{\bU}{{\mathbf{U}}}
\nc{\bV}{{\mathbf{V}}}
\nc{\bW}{{\mathbf{W}}}
\nc{\bX}{{\mathbf{X}}}
\nc{\bY}{{\mathbf{Y}}}
\nc{\bZ}{{\mathbf{Z}}}

\nc{\ba}{{\mathbf{a}}}
\nc{\bb}{{\mathbf{b}}}
\nc{\bc}{{\mathbf{c}}}
\nc{\bd}{{\mathbf{d}}}
\nc{\be}{{\mathbf{e}}}
\nc{\bg}{{\mathbf{g}}}
\nc{\bh}{{\mathbf{h}}}
\nc{\bi}{{\mathbf{i}}}
\nc{\bj}{{\mathbf{j}}}
\nc{\bk}{{\mathbf{k}}}
\nc{\bl}{{\mathbf{l}}}
\nc{\bm}{{\mathbf{m}}}
\nc{\bn}{{\mathbf{n}}}
\nc{\bo}{{\mathbf{o}}}
\nc{\bp}{{\mathbf{p}}}
\nc{\bq}{{\mathbf{q}}}
\nc{\br}{{\mathbf{r}}}
\nc{\bs}{{\mathbf{s}}}
\nc{\bt}{{\mathbf{t}}}
\nc{\bu}{{\mathbf{u}}}
\nc{\bv}{{\mathbf{v}}}
\nc{\bw}{{\mathbf{w}}}
\nc{\bx}{{\mathbf{x}}}
\nc{\by}{{\mathbf{y}}}
\nc{\bz}{{\mathbf{z}}}

\nc{\fA}{{\mathfrak{A}}}
\nc{\fB}{{\mathfrak{B}}}
\nc{\fC}{{\mathfrak{C}}}
\nc{\fD}{{\mathfrak{D}}}
\nc{\fE}{{\mathfrak{E}}}
\nc{\fF}{{\mathfrak{F}}}
\nc{\fG}{{\mathfrak{G}}}
\nc{\fH}{{\mathfrak{H}}}
\nc{\fI}{{\mathfrak{I}}}
\nc{\fJ}{{\mathfrak{J}}}
\nc{\fK}{{\mathfrak{K}}}
\nc{\fL}{{\mathfrak{L}}}
\nc{\fM}{{\mathfrak{M}}}
\nc{\fN}{{\mathfrak{N}}}
\nc{\fO}{{\mathfrak{O}}}
\nc{\fP}{{\mathfrak{P}}}
\nc{\fQ}{{\mathfrak{Q}}}
\nc{\fR}{{\mathfrak{R}}}
\nc{\fS}{{\mathfrak{S}}}
\nc{\fT}{{\mathfrak{T}}}
\nc{\fU}{{\mathfrak{U}}}
\nc{\fV}{{\mathfrak{V}}}
\nc{\fW}{{\mathfrak{W}}}
\nc{\fX}{{\mathfrak{X}}}
\nc{\fY}{{\mathfrak{Y}}}
\nc{\fZ}{{\mathfrak{Z}}}

\nc{\fa}{{\mathfrak{a}}}
\nc{\fb}{{\mathfrak{b}}}
\nc{\fc}{{\mathfrak{c}}}
\nc{\fd}{{\mathfrak{d}}}
\nc{\fe}{{\mathfrak{e}}}
\nc{\ff}{{\mathfrak{f}}}
\nc{\fg}{{\mathfrak{g}}}
\nc{\fh}{{\mathfrak{h}}}
\nc{\fj}{{\mathfrak{j}}}
\nc{\fk}{{\mathfrak{k}}}
\nc{\fl}{{\mathfrak{l}}}
\nc{\fm}{{\mathfrak{m}}}
\nc{\fn}{{\mathfrak{n}}}
\nc{\fo}{{\mathfrak{o}}}
\nc{\fp}{{\mathfrak{p}}}
\nc{\fq}{{\mathfrak{q}}}
\nc{\fr}{{\mathfrak{r}}}
\nc{\fs}{{\mathfrak{s}}}
\nc{\ft}{{\mathfrak{t}}}
\nc{\fu}{{\mathfrak{u}}}
\nc{\fv}{{\mathfrak{v}}}
\nc{\fw}{{\mathfrak{w}}}
\nc{\fx}{{\mathfrak{x}}}
\nc{\fy}{{\mathfrak{y}}}
\nc{\fz}{{\mathfrak{z}}}

\nc{\sA}{{\mathsf{A}}}
\nc{\sB}{{\mathsf{B}}}
\nc{\sC}{{\mathsf{C}}}
\nc{\sD}{{\mathsf{D}}}
\nc{\sE}{{\mathsf{E}}}
\nc{\sF}{{\mathsf{F}}}
\nc{\sG}{{\mathsf{G}}}
\nc{\sH}{{\mathsf{H}}}
\nc{\sI}{{\mathsf{I}}}
\nc{\sJ}{{\mathsf{J}}}
\nc{\sK}{{\mathsf{K}}}
\nc{\sL}{{\mathsf{L}}}
\nc{\sM}{{\mathsf{M}}}
\nc{\sN}{{\mathsf{N}}}
\nc{\sO}{{\mathsf{O}}}
\nc{\sP}{{\mathsf{P}}}
\nc{\sQ}{{\mathsf{Q}}}
\nc{\sR}{{\mathsf{R}}}
\nc{\sS}{{\mathsf{S}}}
\nc{\sT}{{\mathsf{T}}}
\nc{\sU}{{\mathsf{U}}}
\nc{\sV}{{\mathsf{V}}}
\nc{\sW}{{\mathsf{W}}}
\nc{\sX}{{\mathsf{X}}}
\nc{\sY}{{\mathsf{Y}}}
\nc{\sZ}{{\mathsf{Z}}}

\nc{\sa}{{\mathsf{a}}}
\nc{\sd}{{\mathsf{d}}}
\nc{\se}{{\mathsf{e}}}
\nc{\sg}{{\mathsf{g}}}
\nc{\sh}{{\mathsf{h}}}
\nc{\si}{{\mathsf{i}}}
\nc{\sj}{{\mathsf{j}}}
\nc{\sk}{{\mathsf{k}}}
\nc{\sm}{{\mathsf{m}}}
\nc{\sn}{{\mathsf{n}}}
\nc{\so}{{\mathsf{o}}}
\nc{\sq}{{\mathsf{q}}}
\nc{\sr}{{\mathsf{r}}}
\nc{\st}{{\mathsf{t}}}
\nc{\su}{{\mathsf{u}}}
\nc{\sv}{{\mathsf{v}}}
\nc{\sw}{{\mathsf{w}}}
\nc{\sx}{{\mathsf{x}}}
\nc{\sy}{{\mathsf{y}}}
\nc{\sz}{{\mathsf{z}}}

\nc{\oA}{{\overline{A}}}
\nc{\oB}{{\overline{B}}}
\nc{\oC}{{\overline{C}}}
\nc{\oD}{{\overline{D}}}
\nc{\oE}{{\overline{E}}}
\nc{\oF}{{\overline{F}}}
\nc{\oG}{{\overline{G}}}
\nc{\oH}{{\overline{H}}}
\nc{\oI}{{\overline{I}}}
\nc{\oJ}{{\overline{J}}}
\nc{\oK}{{\overline{K}}}
\nc{\oL}{{\overline{L}}}
\nc{\oM}{{\overline{M}}}
\nc{\oN}{{\overline{N}}}
\nc{\oO}{{\overline{O}}}
\nc{\oP}{{\overline{P}}}
\nc{\oQ}{{\overline{Q}}}
\nc{\oR}{{\overline{R}}}
\nc{\oS}{{\overline{S}}}
\nc{\oT}{{\overline{T}}}
\nc{\oU}{{\overline{U}}}
\nc{\oV}{{\overline{V}}}
\nc{\oW}{{\overline{W}}}
\nc{\oX}{{\overline{X}}}
\nc{\oY}{{\overline{Y}}}
\nc{\oZ}{{\overline{Z}}}

\nc{\oa}{{\overline{a}}}
\nc{\ob}{{\overline{b}}}
\nc{\oc}{{\overline{c}}}
\nc{\od}{{\overline{d}}}
\nc{\of}{{\overline{f}}}
\nc{\og}{{\overline{g}}}
\nc{\oh}{{\overline{h}}}
\nc{\oi}{{\overline{i}}}
\nc{\oj}{{\overline{j}}}
\nc{\ok}{{\overline{k}}}
\nc{\ol}{{\overline{l}}}
\nc{\om}{{\overline{m}}}
\nc{\on}{{\overline{n}}}
\nc{\oo}{{\overline{o}}}
\nc{\op}{{\overline{p}}}
\nc{\oq}{{\overline{q}}}
\nc{\os}{{\overline{s}}}
\nc{\ot}{{\overline{t}}}
\nc{\ou}{{\overline{u}}}
\nc{\ov}{{\overline{v}}}
\nc{\ow}{{\overline{w}}}
\nc{\ox}{{\overline{x}}}
\nc{\oy}{{\overline{y}}}
\nc{\oz}{{\overline{z}}}

\nc{\tA}{{\tilde{A}}}
\nc{\tB}{{\tilde{B}}}
\nc{\tC}{{\tilde{C}}}
\nc{\tD}{{\tilde{D}}}
\nc{\tE}{{\tilde{E}}}
\nc{\tF}{{\tilde{F}}}
\nc{\tG}{{\tilde{G}}}
\nc{\tH}{{\tilde{H}}}
\nc{\tI}{{\tilde{I}}}
\nc{\tJ}{{\tilde{J}}}
\nc{\tK}{{\tilde{K}}}
\nc{\tL}{{\tilde{L}}}
\nc{\tM}{{\tilde{M}}}
\nc{\tN}{{\tilde{N}}}
\nc{\tO}{{\tilde{O}}}
\nc{\tP}{{\tilde{P}}}
\nc{\tQ}{{\tilde{Q}}}
\nc{\tR}{{\tilde{R}}}
\nc{\tS}{{\tilde{S}}}
\nc{\tT}{{\tilde{T}}}
\nc{\tU}{{\tilde{U}}}
\nc{\tV}{{\tilde{V}}}
\nc{\tW}{{\tilde{W}}}
\nc{\tX}{{\tilde{X}}}
\nc{\tY}{{\tilde{Y}}}
\nc{\tZ}{{\tilde{Z}}}

\nc{\tfD}{{\tilde{\fD}}}
\nc{\tcA}{{\tilde{\cA}}}
\nc{\tcB}{{\tilde{\cB}}}
\nc{\tcC}{{\tilde{\cC}}}
\nc{\tcD}{{\tilde{\cD}}}
\nc{\tcE}{{\tilde{\cE}}}
\nc{\tcF}{{\tilde{\cF}}}
\nc{\tcM}{{\tilde{\cM}}}
\nc{\tcP}{{\tilde{\cP}}}
\nc{\tcT}{{\tilde{\cT}}}
\nc{\tcW}{{\widetilde{\cW}}}

\nc{\bcT}{\bar{\cT}}

\nc{\tphi}{{\tilde{\varphi}}}

\nc{\ta}{{\tilde{a}}}
\nc{\tb}{{\tilde{b}}}
\nc{\tc}{{\tilde{c}}}
\nc{\td}{{\tilde{d}}}
\nc{\te}{{\tilde{e}}}
\nc{\tf}{{\tilde{f}}}
\nc{\tg}{{\tilde{g}}}
\nc{\ti}{{\tilde{\imath}}}
\nc{\tj}{{\tilde{j}}}
\nc{\tk}{{\tilde{k}}}
\nc{\tl}{{\tilde{l}}}
\nc{\tm}{{\tilde{m}}}
\nc{\tn}{{\tilde{n}}}
\nc{\tp}{{\tilde{p}}}
\nc{\tq}{{\tilde{q}}}
\nc{\tr}{{\tilde{r}}}
\nc{\ts}{{\tilde{s}}}
\nc{\tu}{{\tilde{u}}}
\nc{\tv}{{\tilde{v}}}
\nc{\tw}{{\tilde{w}}}
\nc{\tx}{{\tilde{x}}}
\nc{\ty}{{\tilde{y}}}
\nc{\tz}{{\tilde{z}}}

\nc{\hA}{{\hat{A}}}
\nc{\hB}{{\hat{B}}}
\nc{\hC}{{\hat{C}}}
\nc{\hD}{{\hat{D}}}
\nc{\hE}{{\hat{E}}}
\nc{\hF}{{\hat{F}}}
\nc{\hG}{{\hat{G}}}
\nc{\hH}{{\hat{H}}}
\nc{\hI}{{\hat{I}}}
\nc{\hJ}{{\hat{J}}}
\nc{\hK}{{\hat{K}}}
\nc{\hL}{{\hat{L}}}
\nc{\hM}{{\hat{M}}}
\nc{\hN}{{\hat{N}}}
\nc{\hO}{{\hat{O}}}
\nc{\hP}{{\hat{P}}}
\nc{\hQ}{{\hat{Q}}}
\nc{\hR}{{\hat{R}}}
\nc{\hS}{{\hat{S}}}
\nc{\hT}{{\hat{T}}}
\nc{\hU}{{\hat{U}}}
\nc{\hV}{{\hat{V}}}
\nc{\hW}{{\hat{W}}}
\nc{\hX}{{\widehat{X}}}
\nc{\hY}{{\hat{Y}}}
\nc{\hZ}{{\hat{Z}}}

\nc{\ha}{{\hat{a}}}
\nc{\hb}{{\hat{b}}}
\nc{\hc}{{\hat{c}}}
\nc{\hd}{{\hat{d}}}
\nc{\he}{{\hat{e}}}
\nc{\hg}{{\hat{g}}}
\nc{\hh}{{\hat{h}}}
\nc{\hi}{{\hat{i}}}
\nc{\hj}{{\hat{j}}}
\nc{\hk}{{\hat{k}}}
\nc{\hl}{{\hat{l}}}
\nc{\hm}{{\hat{m}}}
\nc{\hn}{{\hat{n}}}
\nc{\ho}{{\hat{o}}}
\nc{\hp}{{\hat{p}}}
\nc{\hq}{{\hat{q}}}
\nc{\hr}{{\hat{r}}}
\nc{\hs}{{\hat{s}}}
\nc{\hu}{{\hat{u}}}
\nc{\hv}{{\hat{v}}}
\nc{\hw}{{\hat{w}}}
\nc{\hx}{{\hat{x}}}
\nc{\hy}{{\hat{y}}}
\nc{\hz}{{\hat{z}}}

\nc{\hcC}{{\widehat{\cC}}}
\nc{\hcT}{{\widehat{\cT}}}

\nc{\eps}{\upepsilon}
\nc{\lan}{\big\langle}
\nc{\ran}{\big\rangle}
\nc{\kk}{{\Bbbk}}
\nc{\io}{\upiota}
\nc{\Kr}{\mathsf{Kr}}
\nc{\cKr}{\mathcal{K}\!\mathit{r}}

\nc{\Dm}{\bD^{-}}
\nc{\Db}{\bD^{\mathrm{b}}}
\nc{\Dbc}{\bD^{\mathrm{b}}_{\mathrm{c}}}
\nc{\Dp}{\bD^{\mathrm{perf}}}
\nc{\Dperf}{\bD^{\mathrm{perf}}}
\nc{\Dqc}{\bD_{\mathrm{qc}}}
\nc{\Du}{\bD}
\nc{\Dsing}{\bD^{\mathrm{sg}}}
\nc{\Dg}{\bD^{\mathrm{sg}}}

\def\ol{\overline}

\nc{\Rn}{\rR_{\mathrm{node}}}
\nc{\Cn}{\cC_{\mathrm{node}}}
\nc{\Dfd}[1]{\bD_{\mathrm{fd}}(#1)}

\def\bw#1#2{\textstyle{\bigwedge\hskip-0.9mm^{#1}}\hskip0.2mm{#2}}

\nc{\xrightiso}[1]{ \xrightarrow[{\ \raisebox{0.5ex}[0ex][0ex]{$\sim$}\ }]{#1} }

\nc{\thick}{\mathbf{thick}}

\DeclareMathOperator{\Ext}{\mathrm{Ext}}

\DeclareMathOperator{\Bl}{\mathrm{Bl}}

\DeclareMathOperator{\Pic}{\mathrm{Pic}}

\DeclareMathOperator{\Sym}{\mathrm{Sym}}

\DeclareMathOperator{\Gr}{\mathrm{Gr}}
\DeclareMathOperator{\OGr}{\mathrm{OGr}}

\DeclareMathOperator{\OFl}{\mathrm{OFl}}

\DeclareMathOperator{\GL}{\mathbf{GL}}

\DeclareMathOperator{\SO}{\mathbf{SO}}
\DeclareMathOperator{\Spin}{\mathbf{Spin}}

\theoremstyle{plain}

\newtheorem{theorem}{Theorem}[section]

\newtheorem{lemma}[theorem]{Lemma}
\newtheorem{proposition}[theorem]{Proposition}
\newtheorem{corollary}[theorem]{Corollary}

\theoremstyle{definition}

\theoremstyle{remark}

\newtheorem{remark}[theorem]{Remark}

\title{Explicit deformation of the horospherical variety of type $\rG_2$}

\author{Alexander Kuznetsov}
\address{{\sloppy
\parbox{0.9\textwidth}{
Algebraic Geometry Section, Steklov Mathematical Institute of Russian Academy of Sciences,\\
8 Gubkin str., Moscow 119991 Russia
}\bigskip}}
\email{akuznet@mi-ras.ru}
\date{}
\thanks{This work was supported by the Russian Science Foundation under grant no. 19-11-00164,\\
\url{https://rscf.ru/en/project/19-11-00164/}.}

\begin{document}

\maketitle

\begin{abstract}
We give two simple geometric constructions of a smooth family of projective varieties
with central fiber isomorphic to the horospherical variety of type~$\rG_2$
and all other fibers isomorphic to the isotropic orthogonal Grassmannian~$\OGr(2,7)$
and discuss briefly the derived category of this family.
\end{abstract}

\section{Introduction}

Let~$\bG$ be a simple algebraic group of Dynkin type~$\rG_2$.
If~$V_1$ and~$V_2$ are the fundamental representations of~$\bG$ of dimensions~$7$ and~$14$ respectively,
the highest weight vector orbits in~$\P(V_1)$ and~$\P(V_2)$ are 
\begin{itemize}
\item 
$X_1 \subset \P(V_1)$ is a smooth quadric of dimension~$5$, and
\item 
$X_2 \subset \P(V_2)$ is the so-called {\sf adjoint variety} of~$\bG$.
\end{itemize}

In fact, $X_2$ can be realized as a subvariety of~$\Gr(2,V_1)$ (see Lemma~\ref{lem:chain} below for details),
and if~$\cU_{X_2}$ is the restriction to~$X_2$ of the tautological rank~$2$ subbundle from the Grassmannian, 
then
\begin{equation}
\label{eq:tx}
\tX \coloneqq \P_{X_2}(\cU_{X_2}) \xrightarrow{\ p_2\ } X_2.
\end{equation}
is the flag variety of~$\bG$;
in particular, the bundle~$\cU_{X_2}$ and the $\P^1$-fibration~$p_2$ are~$\bG$-equivariant.

Similarly, there is a $\bG$-equivariant vector bundle~$\cC_{X_1}$ of rank~$2$ on the quadric~$X_1$ such that
\begin{equation}
\label{eq:tx-x1}
\tX \cong \P_{X_1}(\cC_{X_1}) \xrightarrow{\ p_1\ } X_1
\end{equation} 
is a $\bG$-equivariant $\P^1$-fibration;
the bundle~$\cC_{X_1}$ is known as the {\sf Cayley bundle}.

The {\sf horospherical variety~$X$ of type~$\rG_2$} can be constructed out of these data in several ways.
We outline below three related constructions, for details see~\cite{Pas,GPPS}. 
Let~$H_1$ and~$H_2$ denote the ample generators of~$\Pic(X_1)$ and~$\Pic(X_2)$; 
abusing notation we will denote in the same way their pullbacks to~$\tX$.
Consider the following projective bundles
\begin{equation*}
\P_{X_1}(\cO_{X_1}(-H_1) \oplus \cC_{X_1}),
\qquad 
\P_{X_2}(\cU_{X_2} \oplus \cO_{X_2}(-H_2)),
\qquad\text{and}\qquad 
\P_{\tX}(\cO_\tX(-H_1) \oplus \cO_\tX(-H_2)),
\end{equation*}
over~$X_1$, $X_2$, and~$\tX$, respectively.
Note that the summands of rank~$1$ induce sections of the first two projective bundles,
and a pair of sections of the last one.
It is not hard to prove that the relative hyperplane class of each of these projective bundles 
is base point free and induces a $\bG$-equivariant morphism into~$\P(V_1 \oplus V_2)$.
The image of each morphism is the horospherical variety~$X \subset \P(V_1 \oplus V_2)$, 
the images of the sections are disjoint subvarieties
\begin{equation*}
X_1 = X \cap \P(V_1)
\qquad\text{and}\qquad 
X_2 = X \cap \P(V_2),
\end{equation*}
and in this way we obtain isomorphisms
\begin{align}
\label{eq:bl-x2-x}
\Bl_{X_2}(X) &\cong \P_{X_1}(\cO_{X_1}(-H_1) \oplus \cC_{X_1}),
\\
\label{eq:bl-x1-x}
\Bl_{X_1}(X) &\cong \P_{X_2}(\cU_{X_2} \oplus \cO_{X_2}(-H_2)),
\\
\label{eq:bl-x1-x2-x}
\Bl_{X_1 \sqcup X_2}(X) &\cong \P_{\tX}(\cO_\tX(-H_1) \oplus \cO_\tX(-H_2)).
\end{align} 
All these isomorphisms are $\bG$-equivariant.

The constructions described above are quite general and can be applied 
to horospherical varieties with Picard number~$1$ of other types 
(see~\cite[Theorem~0.1]{Pas} for classification).
The special property of the horospherical variety of type~$\rG_2$ 
is its relation to another smooth projective~$\bG$-variety ---
the {\sf orthogonal isotropic Grassmannian}
\begin{equation}
\label{eq:y}
Y = \OGr(2,V_1),
\end{equation}
the subvariety of~$\Gr(2,V_1)$ that parameterizes 2-dimensional subspaces 
isotropic with respect to the quadratic equation of~$X_1 \subset \P(V_1)$.

It can be observed that~$X$ and~$Y$ share all numerical invariants;
in particular, they have the same ranks of the Grothendieck groups (equal to~$12$),
the same Fano indices (equal to~$4$), the same dimensions of the spaces of global sections of~$\cO(1)$ (equal to~$21$), and so on.
A nice explanation to these coincidences was given in~\cite[Proposition~2.3]{PP}, 
where a \emph{smooth degeneration} of~$Y$ to~$X$,
i.e., a smooth projective variety over~$\AA^1$ with central fiber isomorphic to~$X$
and all other fibers isomorphic to~$Y$ has been constructed.

The goal of this paper is to give two geometric constructions of such a family over any smooth pointed curve~$(C,0)$.

\begin{theorem}
\label{thm:main}
Let~$(C,0)$ be a smooth pointed curve and set~$\cY \coloneqq Y \times C$. 
Then there is a commutative diagram
\begin{equation}
\label{diagram}
\vcenter{\xymatrix@C=0em{
&&&&
\Bl_{X_1}(\cX) \ar@{=}[rr] \ar[dl]_{\pi_\cX} &&
\Bl_{X_2}(\cY) \ar[dr]^{\pi_\cY}
\\
X_1 \ar@{^{(}->}[rrr] &&&
\cX \ar[drr]_{f_\cX} &&&&
\cY \ar[dll]^{f_\cY} &&&
X_2 \ar@{_{(}->}[lll] 
\\
&&&&& C
}}
\end{equation}
where
\begin{itemize}
\item 
$f_\cX \colon \cX \to C$ is a smooth projective morphism such that~$\cX_0 \coloneqq f_\cX^{-1}(0) \cong X$, 
\item 
$f_\cY \colon \cY \to C$ is the projection to the second factor so that~$\cY_0 \coloneqq f_\cY^{-1}(0) \cong Y$, and
\item 
$\pi_\cX$ and~$\pi_\cY$ are the blowups of the smooth subvarieties~$X_1 \subset X = \cX_0 \subset \cX$ 
and~$X_2 \subset Y = \cY_0 \subset \cY$.
\end{itemize}
In particular, $f_{\cX}^{-1}(C \setminus \{0\}) \cong f_{\cY}^{-1}(C \setminus \{0\}) \cong Y \times (C \setminus \{0\})$,
so that~$X$ is a smooth degeneration of~$Y$.
\end{theorem}

\begin{theorem}
\label{thm:main-2}
There is a vector bundle~$\tcW$ of rank~$3$ on~$X_1 \times C$ and a commutative diagram
\begin{equation}
\label{diagram-2}
\vcenter{\xymatrix{
& \tX \times C \ar@{^{(}->}[d] \ar@/_2em/[ddl]_{p_1} \ar@/^2em/[ddr]^{p_2}
\\
& \P_{X_1 \times C}(\tcW) \ar[dl] \ar[d]^\uprho 
\\
X_1 \times C &
\cX &
X_2 \times C \ar@{_{(}->}[l]
}}
\end{equation}
over~$C$, where~$\cX \to C$ is the same as in Theorem~\textup{\ref{thm:main}}
and~$\uprho$ is the blowup of~$X_2 \times C \subset \cX$.
\end{theorem}

The crucial observation (Proposition~\ref{prop:bl-x2-y}) on which the proof of both theorems relies 
is that there is a natural embedding~$X_2 \hookrightarrow Y$ 
and that the blowup~$\Bl_{X_2}(Y)$ has a structure of a projective bundle over~$X_1$,
analogous to~\eqref{eq:bl-x2-x}.

\section{The key observation}

Recall that~$V_1$ denotes the fundamental $7$-dimensional representation of the group~$\bG$ and~$Y$ is defined in~\eqref{eq:y}.
We denote by~$\cU \subset V_1 \otimes \cO$ and~$\cU^\perp \subset V_1^\vee \otimes \cO$ 
the tautological subbundles of rank~$2$ and~$5$ on~$\Gr(2,V_1)$, respectively,
and write~$\cO(1)$ for the Pl\"ucker line bundle.
We also denote by~$\cS$ the spinor bundle on~$Y$ 
with the convention opposite to that of~\cite[\S6]{K08}, so that~$\wedge^2\cS^\vee \cong \cO(1)$.

\begin{lemma}
\label{lem:chain}
There exists a chain of embeddings
\begin{equation*}
X_2 \hookrightarrow Y \hookrightarrow \Gr(2,V_1) \subset \P(\wedge^2V_1),
\end{equation*}
such that
\begin{itemize}
\item 
$Y \subset \Gr(2,V_1)$ is the zero locus of a regular section of~$\Sym^2\cU^\vee$, 
\item 
$X_2 \subset \Gr(2,V_1)$ is the zero locus of a regular section of~$\cU^\perp(1)$, and
\item 
$X_2 \subset Y$ is the zero locus of a regular section of~$\cS^\vee$. 
\end{itemize}
Moreover, the restrictions to~$X_2$ of~$\cU$ and~$\cS$ are isomorphic.
\end{lemma}

\begin{proof}
The description of~$X_2$ as the zero locus in~$\Gr(2,V_1)$ was discovered by Mukai 
(the global section of~$\cU^\perp(1)$ is given by a 3-form~$\lambda \in \wedge^3V_1^\vee$ 
whose stabilizer is the group~$\bG$).
The restriction of~$\cO(1)$ to~$X_2$ generates~$\Pic(X_2)$,
the linear span of~$X_2$ in~$\P(\wedge^2V_1)$ is~$\P(V_2)$,
and the corresponding embedding~$V_2 \hookrightarrow \wedge^2V_1$ extends to an exact sequence
\begin{equation}
\label{eq:v1v2}
0 \xrightarrow\quad V_2 \xrightarrow\quad \wedge^2V_1 \xrightarrow{\ \lambda\ } V_1^\vee \xrightarrow\quad 0
\end{equation}
of representations of~$\bG$.
Note that the~$\bG$-action on~$V_1$ preserves a nondegenerate quadratic form 
(the equation of the quadric~$X_1 \subset \P(V_1)$), 
hence we have a chain of group embeddings
\begin{equation*}
\bG \subset \SO(V_1) \subset \GL(V_1).
\end{equation*}
Moreover, a highest weight vector of~$V_2$ with respect to~$\bG$ 
is also a highest weight vector of~$\wedge^2V_1$ with respect to~$\SO(V_1)$ and~$\GL(V_1)$, 
and therefore we have a chain of the highest weight vector orbits
\begin{equation*}
X_2 \subset Y \subset \Gr(2,V_1)
\end{equation*}
of respective groups.

The description of~$Y \subset \Gr(2,V_1)$ as the zero locus is tautological 
(the section corresponds to the quadratic form on~$V_1$ preserved by~$\SO(V_1)$),
and the description of~$X_2 \subset Y$ as the zero locus of a section of~$\cS^\vee$
and an isomorphism~$\cS\vert_{X_2} \cong \cU\vert_{X_2}$
are established in~\cite[Lemma~8.3]{K06}.
\end{proof}

\begin{corollary}
\label{cor:y-cap-pv2}
There is an equality~$X_2 = Y \cap \P(V_2)$ of schemes.
\end{corollary}

The intersection in the right-hand side is, of course, highly non-transverse.

\begin{proof}
Let~$\cI_{X_2}$ be the ideal of~$X_2 \subset Y = \OGr(2,V_1)$.
By Lemma~\ref{lem:chain} we have an exact sequence
\begin{equation*}
0 \to \cO_Y \to \cS^\vee \to \cI_{X_2}(1) \to 0.
\end{equation*}
The space of global sections of~$\cS^\vee$ on~$Y$ is the spinor $8$-dimensional representation~$\SS$ of~$\Spin(V_1)$;
when restricted to~$\bG$ it is isomorphic to the direct sum~$V_1 \oplus \kk$.
Since, moreover, $\cS^\vee$ is globally generated, the above exact sequence gives an epimorphism~$V_1 \otimes \cO_Y \twoheadrightarrow \cI_{X_2}(1)$.
This means that~$X_2 \subset \Gr(2,V_1)$ is scheme-theoretically cut out by the hyperplanes
corresponding to the subspace
\begin{equation*}
V_1 \subset H^0(Y, \cO_Y(1)) = \wedge^2V_1^\vee.
\end{equation*}
This embedding is obviously $\bG$-equivariant, hence is given by the dual of the second map in~\eqref{eq:v1v2}, 
hence~$X_2 = Y \cap \P(V_2)$.
\end{proof}

The following proposition is an analogue of~\eqref{eq:bl-x2-x}; 
it is the key to the proof of Theorems~\ref{thm:main} and~\ref{thm:main-2}.

\begin{proposition}
\label{prop:bl-x2-y}
There is a $\bG$-equivariant isomorphism
\begin{equation}
\label{eq:bl-x2-y}
\Bl_{X_2}(Y) \cong \P_{X_1}(\cW_{X_1}),
\end{equation}
where~$\cW_{X_1}$ is a $\bG$-equivariant vector bundle on~$X_1$ of rank~$3$;
it fits into an exact sequence
\begin{equation}
\label{eq:cayley}
0 \to \cC_{X_1} \to \cW_{X_1} \to \cO_{X_1}(-H_1) \to 0.
\end{equation} 
\end{proposition}

\begin{proof}
Consider the orthogonal isotropic partial flag variety and its two projections
\begin{equation*}
\xymatrix{
& \OFl(2,3;V_1) \ar[dl] \ar[dr] 
\\
\OGr(2,V_1) &&
\OGr(3,V_1)
}
\end{equation*}
The fibers of the first are nondegenerate conics; 
in fact, it is a $\P^1$-bundle, 
the corresponding vector bundle on~$\OGr(2,V_1)$ is precisely the spinor bundle~$\cS$ 
that was mentioned above, hence we have an isomorphism
\begin{equation*}
\OFl(2,3;V_1) \cong \P_{\OGr(2,V_1)}(\cS) = \P_Y(\cS).
\end{equation*}
Similarly, the second map is a $\P^2$-bundle.
We denote the corresponding $\SO(V_1)$-equivariant vector bundle of rank~$3$ on~$\OGr(3,V_1)$ by~$\cW$,
so that we have an isomorphism
\begin{equation}
\label{eq:ofl}
\OFl(2,3;V_1) \cong \P_{\OGr(3,V_1)}(\cW).
\end{equation}

Note that~$\OGr(3,V_1)$ is a smooth $6$-dimensional quadric in the projectivization
of the $8$-dimensional spinor representation~$\SS$ of~$\Spin(V_1)$ (the universal covering of~$\SO(V_1)$).
The restriction of this representation to the group~$\bG$ splits as~\mbox{$\SS = V_1 \oplus \kk$},
and the hyperplane section of~$\OGr(3,V_1) \subset \P(\SS)$ by~$\P(V_1)$ is the quadric~$X_1$.
Its preimage 
\begin{equation*}
\OFl(2,3;V_1) \times_{\OGr(3,V_1)} X_1 \cong \P_{X_1}(\cW\vert_{X_1})
\end{equation*}
is a relative (over~$Y = \OGr(2,V_1)$) hyperplane section of~$\P_Y(\cS)$;
therefore, it is isomorphic to the blowup of~$Y$ along the zero locus of the corresponding section of~$\cS^\vee$.
By Lemma~\ref{lem:chain} this zero locus is~$X_2 \subset Y$, 
hence we obtain the required isomorphism~\eqref{eq:bl-x2-y}, where~$\cW_{X_1} \coloneqq \cW\vert_{X_1}$.

To construct exact sequence~\eqref{eq:cayley} note that by Lemma~\ref{lem:chain} 
the normal bundle of~$X_2 \subset Y$ is
\begin{equation*}
\cN_{X_2/Y} \cong \cS\vert_{X_2}^\vee \cong \cU_{X_2}^\vee,
\end{equation*}
hence, using~\eqref{eq:tx} and~\eqref{eq:tx-x1}, we deduce that the exceptional divisor of the blowup is
\begin{equation*}
\P_{X_2}(\cN_{X_2/Y}) \cong 
\P_{X_2}(\cU^\vee_{X_2}) \cong 
\P_{X_2}(\cU_{X_2}) = 
\tX = 
\P_{X_1}(\cC_{X_1}).
\end{equation*}
Moreover, the induced embedding~$\P_{X_1}(\cC_{X_1}) \hookrightarrow \P_{X_1}(\cW_{X_1})$ 
is compatible with the projection to~$X_1$, 
and also with the relative hyperplane classes, 
hence it induces an embedding of vector bundles~$\cC_{X_1} \hookrightarrow \cW_{X_1}$.
The quotient is a line bundle, so it can be identified with~$\cO_{X_1}(-H_1)$ by the determinant computation,
taking into account the isomorphisms
\begin{equation*}
\det(\cC_{X_1}) \cong \cO_{X_1}(-3H_1)
\qquad\text{and}\qquad 
\det(\cW_{X_1}) \cong \cO_{X_1}(-4H_1),
\end{equation*}
that follow from~\eqref{eq:tx-x1} and~\eqref{eq:ofl} by a canonical class computation.
\end{proof}

\begin{remark}
\label{rem:nonsplit}
The crucial difference between~\eqref{eq:bl-x2-y} and~\eqref{eq:bl-x2-x} is that the extension~\eqref{eq:cayley}
defining the vector bundle~$\cW_{X_1}$ is \emph{non-trivial}. 
One way to see this is the following: 
if~\eqref{eq:cayley} were split then the embedding~$\cO_{X_1}(-H_1) \hookrightarrow \cW_{X_1}$
would give an embedding~$X_1 \hookrightarrow Y$ of the $5$-dimensional quadric~$X_1$, 
but it is well known that~$Y = \OGr(2,V_1)$ (and even~$\Gr(2,V_1)$ containing it)
does not contain quadrics of dimension greater than~$4$.
\end{remark}

\section{Proof of Theorem~\ref{thm:main}}

Recall that~$\cY = Y \times C$ and the map~$f_\cY \colon \cY \to C$ is the projection. 
Consider the subvariety
\begin{equation*}
X_2 \hookrightarrow Y = \cY_0 \hookrightarrow \cY
\end{equation*}
(where recall that~$\cY_0 \subset \cY$ denotes the central fiber)
and the blowup~$\pi_\cY \colon \Bl_{X_2}(\cY) \to \cY$.
This gives us the right half of the diagram~\eqref{diagram}.
To construct the left half we need two lemmas.

\begin{lemma}
\label{lem:divisors}
The scheme central fiber of~$\Bl_{X_2}(\cY) \xrightarrow{\ \pi_\cY\ } \cY \xrightarrow{\ f_\cY\ } C$ 
is the normal crossing divisor
\begin{equation*}
\P_{X_2}(\cU_{X_2} \oplus \cO_{X_2}(-H_2)) \ \bigcup_\tX\ \P_{X_1}(\cW_{X_1})
\end{equation*}
where the first component is the exceptional divisor of~$\pi_\cY$ 
and the second is the strict transform of~$\cY_0$.
\end{lemma}

\begin{proof}
Since~$\cY$ is the product~$Y \times C$, using Lemma~\ref{lem:chain} we compute the normal bundle
\begin{equation*}
\cN_{X_2/\cY} \cong \cN_{X_2/Y} \oplus \cO_{X_2} \cong \cU_{X_2}^\vee \oplus \cO_{X_2}.
\end{equation*}
This is a twist of~$\cU_{X_2} \oplus \cO_{X_2}(-H_2)$, 
hence we obtain the description of the first component of the central fiber of~$\Bl_{X_2}(\cY)$.
The second component is isomorphic to the blowup~$\Bl_{X_2}(Y)$, so Proposition~\ref{prop:bl-x2-y} applies.
Finally, the intersection of the components is the projectivization of~$\cN_{X_2/Y} \cong \cU_{X_2}^\vee$,
hence~\eqref{eq:tx} shows that it is isomorphic to~$\tX$.
\end{proof}

Now consider the trivial vector bundle~$\wedge^2V_1 \otimes \cO_C$ and filtration~\eqref{eq:v1v2} on its central fiber.
It induces a vector bundle~$\cV$ on~$C$ and a morphism~$\alpha \colon \wedge^2V_1 \otimes \cO_C \to \cV$ 
that fit into an exact sequence
\begin{equation*}
0 \xrightarrow\quad 
\wedge^2V_1 \otimes \cO_C \xrightarrow{\ \alpha\ } 
\cV \xrightarrow\quad 
V_2 \otimes \cO_{\{0\}} \xrightarrow\quad 0,
\end{equation*}
where~$\cO_{\{0\}}$ is the structure sheaf of the point~$\{0\} \in C$;
the central fiber of~$\cV$ is canonically an extension
\begin{equation}
\label{eq:v1v2-dual}
0 \xrightarrow\quad V_1^\vee \xrightarrow\quad \cV_{\{0\}} \xrightarrow\quad V_2 \xrightarrow\quad 0
\end{equation}
(opposite to~\eqref{eq:v1v2}), and the morphism~$\alpha_{\{0\}}$ 
factors as~$\wedge^2V_1 \xrightarrow{\ \lambda\ } V_1^\vee \xrightarrow\quad \cV_{\{0\}}$.
Note that the morphism~$\alpha$ induces a birational map~$\P_{C}(\wedge^2V_1 \otimes \cO_C) \dashrightarrow \P_{C}(\cV)$
of projective bundles over~$C$.

\begin{lemma}
\label{lem:et}
Consider the embeddings
\begin{equation*}
\P(V_2) \subset \P(\wedge^2V_1) \hookrightarrow \P_{C}(\wedge^2V_1 \otimes \cO_C)
\qquad\text{and}\qquad 
\P(V_1^\vee) \subset \P(\cV_{\{0\}}) \hookrightarrow \P_{C}(\cV)
\end{equation*}
into the central fibers of the projective bundles.
The birational map~$\alpha$ induces an isomorphism of blowups
\begin{equation*}
\Bl_{\P(V_2)}(\P_{C}(\wedge^2V_1 \otimes \cO_C)) \cong \Bl_{\P(V_1^\vee)}(\P_{C}(\cV))
\end{equation*}
over~$C$ such that the exceptional divisor of each side 
coincides with the strict transform of the central fiber of the projective bundle of the other side.
\end{lemma}

\begin{proof}
This is an elementary transformation of projective bundles,
so the argument is standard.
\end{proof}

Now we construct the left half of~\eqref{diagram}.
Consider the natural embedding
\begin{equation*}
\cY = 
Y \times C =
\OGr(2,V_1) \times C \hookrightarrow 
\P(\wedge^2V_1) \times C = 
\P_{C}(\wedge^2V_1 \otimes \cO_C).
\end{equation*}
By Corollary~\ref{cor:y-cap-pv2} the strict transform of~$\cY$ 
under the blowup~$\Bl_{\P(V_2)}(\P_{C}(\wedge^2V_1 \otimes \cO_C))$ of Lemma~\ref{lem:et}
is isomorphic to~$\Bl_{X_2}(\cY)$.
Consider the composition
\begin{equation*}
\Bl_{X_2}(\cY) \hookrightarrow 
\Bl_{\P(V_2)}(\P_{C}(\wedge^2V_1 \otimes \cO_C)) \cong
\Bl_{\P(V_1^\vee)}(\P_{C}(\cV)) \to 
\P_{C}(\cV)
\end{equation*}
of the induced embedding with the isomorphism from Lemma~\ref{lem:et} and the obvious contraction.
We denote by~$\cX \subset \P_{C}(\cV)$ its image, and consider the obtained morphisms
\begin{equation*}
\Bl_{X_2}(\cY) \xrightarrow{\ \pi_\cX\ } \cX \xrightarrow{\ f_\cX\ } C.
\end{equation*}
It remains to prove that~$f_\cX$ is smooth, its central fiber is isomorphic to~$X$,
and~$\pi_\cX$ is the blowup of~$X_1 \subset \cX$.

By Lemma~\ref{lem:et} the morphism~$\Bl_{\P(V_2)}(\P_{C}(\wedge^2V_1 \otimes \cO_C)) \to \P_{C}(\cV)$
contracts the strict transform of the central fiber of~$\P_{C}(\wedge^2V_1 \otimes \cO_C)$
and is an isomorphism on its complement, 
it follows that~$\pi_\cX$ contracts the strict transform of the central fiber of~$\cY$
and is an isomorphism on its complement.

The restriction of~$\pi_\cX$ to the exceptional divisor~$\P_{X_2}(\cU_{X_2} \oplus \cO_{X_2}(-H_2))$ of~$\pi_\cY$
is the morphism given by the relative hyperplane class, hence its image is the horospherical variety~$X$.
This is the scheme central fiber of~$f_\cX$, so smoothness of~$X$ implies that~$f_\cX$ is smooth.

The restriction of~$\pi_\cX$ to the strict transform~$\Bl_{X_2}(Y) \cong \P_{X_1}(\cW_{X_1})$
coincides by construction with the morphism from Proposition~\ref{prop:bl-x2-y},
therefore~$\pi_\cX(\Bl_{X_2}(Y)) = X_1 \subset X$.

Finally, the fact that~$\pi_\cX$ is the blowup of~$X_1 \subset X = \cX_0 \subset \cX$ follows from~\cite[Lemma~2.5]{K18}.

\section{Proof of Theorem~\ref{thm:main-2}}

Recall the exact sequence~\eqref{eq:cayley} and let 
\begin{equation*}
\upepsilon \in \Ext^1(\cO_{X_1}(-H_1), \cC_{X_1})
\end{equation*}
denote its extension class.
Recall that~$\upepsilon \ne 0$ by Remark~\ref{rem:nonsplit}.

Let~$\cL$ be the line bundle of degree~$1$ on~$C$ associated with the point~$\{0\} \in C$ 
and let~$s_0 \in \rH^0(C,\cL)$ be the corresponding global section.
We define a vector bundle~$\tcW$ on~$X_1 \times C$ as an extension
\begin{equation}
\label{eq:cw}
0 \to \cC_{X_1} \boxtimes \cL \to \tcW \to \cO_{X_1}(-H_1) \boxtimes \cO_{C} \to 0
\end{equation}
with extension class
\begin{equation*}
\upepsilon \otimes s_0 \in 
\Ext^1(\cO_{X_1}(-H_1), \cC_{X_1}) \otimes \rH^0(C, \cL) \cong
\Ext^1(\cO_{X_1}(-H_1) \boxtimes \cO_{C}, \cC_{X_1} \boxtimes \cL).
\end{equation*}
Thus, over~$\{0\}$ the extension splits, so that
\begin{align}
\label{eq:cw-0}
\tcW\vert_{X_1 \times \{0\}} &\cong \cO_{X_1}(-H_1) \oplus \cC_{X_1},\\
\intertext{while for each~$0 \ne t \in C$ the extension is isomorphic to~\eqref{eq:cayley}, so that}
\label{eq:cw-t}
\tcW\vert_{X_1 \times (C \setminus \{0\})} &\cong \cW_{X_1} \boxtimes \cO_{C \setminus \{0\}}.
\end{align}

Now consider the projective bundle~$\P_{X_1 \times C}(\tcW)$
and its relative hyperplane class~$H$.
Since both vector bundles~$\cC_{X_1}^\vee$ and~$\cO_{X_1}(H_1)$ are globally generated,
the linear system~$|H|$ is base point free on each fiber over~$C$,
therefore it defines a morphism
\begin{equation*}
\P_{X_1 \times C}(\tcW) \to \P_{C}(\cV')
\end{equation*}
to an appropriate projective bundle over~$C$ 
(in fact, this bundle can be identified with the bundle~$\P_C(\cV)$ constructed in the proof of Theorem~\ref{thm:main}).
We denote the image by~$\cX$ and claim that it is smooth over~$C$ 
with fibers~$X$ and~$Y$ over~$\{0\} \in C$ and~$C \setminus \{0\}$, respectively, and that
\begin{equation*}
\P_{X_1 \times C}(\tcW) \cong \Bl_{X_2 \times C}(\cX).
\end{equation*}
Indeed, the fiber~$\cX_t$ of~$\cX$ over a point~$t \in C$ is the image of~$\P_{X_1}(\tcW_t)$
under the morphism given by the relative hyperplane class.
When~$t = 0$ by~\eqref{eq:cw-0} this agrees with the definition~\eqref{eq:bl-x2-x} of the horospherical variety, so that
\begin{equation*}
\cX_0 \cong X.
\end{equation*}
On the other hand, for~$t \ne 0$ we apply~\eqref{eq:cw-t} and Proposition~\ref{prop:bl-x2-y} and deduce that
\begin{equation*}
\cX_t \cong Y.
\end{equation*}
Finally, note that the exceptional locus of the morphism~$\uprho \colon \P_{X_1 \times C}(\tcW) \to \cX$
is the projective subbundle
\begin{equation*}
\P_{X_1 \times C}(\cC_{X_1} \boxtimes \cL) \cong
\P_{X_1}(\cC_{X_1}) \times C \cong
\tX \times C \cong 
\P_{X_2}(\cU_{X_2}) \times C 
\end{equation*}
and it is contracted by~$\uprho$ onto the subvariety~$X_2 \times C \subset \cX$.

\begin{remark}
One can obtain the vector bundle~$\tcW$ on~$X_1 \times C$ 
from the (trivial over~$C$) vector bundle~\mbox{$\cW_{X_1} \boxtimes \cO_{C}$}
and the filtration~\eqref{eq:cayley} of its central fiber 
by an elementary transformation, analogous to that of Lemma~\ref{lem:et}.
Using this one can merge the constructions of Theorem~\ref{thm:main} and~\ref{thm:main-2}.
\end{remark}

\section{Derived categories}

There are several ways in which the constructions of Theorems~\ref{thm:main} and~\ref{thm:main-2} can be applied.
For instance, one can relate derived categories of~$X$ and~$Y$.
Recall that both have a full exceptional collection:
in the case of~$X$ that has been proved in~\cite[Theorem~8.20]{GPPS} 
and in the case of~$Y$ in~\cite[Theorem~7.1]{K08};
moreover, \cite[Remark~8.22]{GPPS} points out that the collections have the same structure.

It turns out that these two collections can be glued.
In fact, one can define a \emph{relative exceptional collection} on~$\cX$ over~$C$,
that will coincide with the collection from~\cite[Theorem~7.1]{K08} over~$C \setminus \{0\}$
and with the collection from~\cite[Theorem~8.20]{GPPS} on the central fiber.

Explicitly, recall the notation of diagram~\eqref{diagram} and denote additionally by
\begin{itemize}
\item 
$i_1 \colon E_1 \hookrightarrow \Bl_{X_1}(\cX)$ the embedding of exceptional divisor of~$\pi_\cX$, and
\item 
$i_2 \colon E_2 \hookrightarrow \Bl_{X_2}(\cY)$ the embedding of exceptional divisor of~$\pi_\cY$.
\end{itemize}
Recall from Lemma~\ref{lem:divisors} that~$E_1 \to X_1$ and~$E_2 \to X_2$ are $\P^2$-bundles
and the intersection
\begin{equation*}
E \coloneqq E_1 \cap E_2 \cong \tX
\end{equation*}
is transverse.
Denote~$\cU_\cY \coloneqq \cU \boxtimes \cO_C$ and~$\cS_\cY \coloneqq \cS \boxtimes \cO_C$.
Then one can check that on~$\Bl_{X_1}(\cX) \cong \Bl_{X_2}(\cY)$ there are distinguished triangles
\begin{equation}
\pi_\cX^*\cS_\cX \to 
\pi_\cY^*\cS_\cY \to 
i_{1*}\cO_{E_1}(-E)
\qquad\text{and}\qquad
\pi_\cY^*\cU_\cY \to 
\pi_\cX^*\cU_\cX \to 
i_{2*}\cO_{E_2}(-H_2 - 2E)
\end{equation} 
defining objects~$\cS_\cX$ and~$\cU_\cX$ in~$\Db(\cX)$. 
Note that~$E_1$ and~$E_2$ are both supported over~$\{0\} \in C$,
hence over~$C \setminus \{0\}$ these triangles simplify to isomorphisms 
between the restrictions of~$\cS_\cX$ and~$\cS_\cY$, and~$\cU_\cX$ and~$\cU_\cY$, respectively.
On the other hand, the restrictions to the central fiber~$\cX_0 \cong X$ can be identified as
\begin{equation*}
\cS_\cX\vert_X \cong \mathbb{U}
\qquad\text{and}\qquad 
\cU_\cX\vert_X \cong \widehat{\mathbb{S}}
\end{equation*}
where the right-hand sides are defined in~\cite[Propositions~8.4, 8.7 and Lemma~8.12]{GPPS}.

One can also prove that there is a $C$-linear semiorthogonal decomposition
\begin{equation*}
\Db(\cX) = \langle \cA, \cA(H), \cA(2H), \cA(3H) \rangle
\end{equation*}
where~$H$ is the relative hyperplane class for~$\cX$ over~$C$ and 
\begin{equation*}
\cA = \langle \cS_\cX \otimes \Db(C), \cU_\cX \otimes \Db(C), \cO_\cX \otimes \Db(C) \rangle.
\end{equation*}
Moreover, after base change to~$\{0\}$ and~$C \setminus \{0\}$ (see~\cite{K11}) these decompositions
coincide with the respective decompositions of~$\Db(X)$ and~$\Db(Y \times (C \setminus \{0\}))$.

\subsection*{Acknowledgement}

I thank Sasha Samokhin for the discussion where Theorem~\ref{thm:main} was discovered
and Jun-Muk Hwang and Nicolas Perrin for useful comments about the preliminary version of this note.

\bibliography{hsg2}
\bibliographystyle{alpha}

\end{document}